\documentclass[11pt,leqno]{article}
\usepackage{hyperref}
\usepackage{soul}
\usepackage{pdfsync}
\usepackage[doc]{optional}
\usepackage{xcolor}
\definecolor{labelkey}{rgb}{0,0.08,0.45}
\definecolor{rekey}{rgb}{0,0.6,0.0}
\definecolor{Brown}{rgb}{0.45,0.0,0.05}
\usepackage{exscale,relsize}
\usepackage{amsmath}
\usepackage{amsfonts}
\usepackage{amssymb}
\usepackage{calc}
\usepackage{theorem}
\usepackage{pifont}      %needed by dingautolist
\usepackage{graphicx}
\usepackage[toc,page]{appendix}
\usepackage{appendix}
\DeclareMathOperator{\weakstarly}{\rightharpoondown_{\mathrm{w*}}}

\newcommand{\scal}[2]{\langle{{#1},{#2}}\rangle}

\newcommand{\RR}{\ensuremath{\mathbb R}}

\newcommand{\RX}{\ensuremath{\,\left]-\infty,+\infty\right]}}

\newcommand{\NN}{\ensuremath{\mathbb N}}

\newcommand{\menge}[2]{\big\{{#1} \mid {#2}\big\}}

\newcommand{\To}{\ensuremath{\rightrightarrows}}

\newcommand{\dom}{\ensuremath{\operatorname{dom}}}

\newcommand{\gra}{\ensuremath{\operatorname{gra}}}

\newcommand{\intdom}{\ensuremath{\operatorname{int}\operatorname{dom}}\,}
\newcommand{\inte}{\ensuremath{\operatorname{int}}}

\newcommand{\conv}{\ensuremath{\operatorname{conv}}}

\renewcommand{\phi}{\ensuremath{\varphi}}

\newtheorem{theorem}{Theorem}[section]
\newtheorem{lemma}[theorem]{Lemma}
\newtheorem{fact}[theorem]{Fact}
\newtheorem{corollary}[theorem]{Corollary}
\newtheorem{proposition}[theorem]{Proposition}

\theoremstyle{plain}{\theorembodyfont{\rmfamily}
}
\theoremstyle{plain}{\theorembodyfont{\rmfamily}
}
\theoremstyle{plain}{\theorembodyfont{\rmfamily}
}
\theoremstyle{plain}{\theorembodyfont{\rmfamily}
}
\theoremstyle{plain}{\theorembodyfont{\rmfamily}
\newtheorem{remark}[theorem]{Remark}}
\newtheorem{problem}[theorem]{Open Problem}
\theoremstyle{plain}{\theorembodyfont{\rmfamily}
}

\begin{document}

%\sffamily

\title{\sffamily{
Maximality of the sum of  the subdifferential operator and a maximally monotone operator}}

\author{
Liangjin\ Yao\thanks{Department of Mathematical and Statistical Sciences,
 University of Alberta
Edmonton, Alberta, T6G 2G1 Canada
E-mail:  \texttt{liangjin@ualberta.ca}.}}

\date{June 29,  2014}
\maketitle

\begin{abstract} \noindent
The most important open problem in Monotone Operator Theory
concerns the maximal monotonicity of the sum of two
maximally monotone operators provided that the classical
Rockafellar's constraint qualification holds, which is called the ``sum problem''.

In this paper, we establish the maximal monotonicity of $A+B$
provided that
$A$ and $B$ are maximally monotone operators such that
$\dom A\cap\inte\dom B\neq\varnothing$, and
$A+N_{\overline{\dom A}}$ is of type (FPV). This generalizes various current results and also gives an affirmative answer to a problem posed by Borwein and Yao.
Moreover, we present an equivalent description of the sum problem.
\end{abstract}



\section{Introduction}

Throughout this paper, we assume that
$X$ is a real Banach space with norm $\|\cdot\|$,
that $X^*$ is the continuous dual of $X$, and
that $X$ and $X^*$ are paired by $\scal{\cdot}{\cdot}$.
Let $A\colon X\To X^*$
be a \emph{set-valued operator}  (also known as  point-to-set mapping or multifunction)
from $X$ to $X^*$, i.e., for every $x\in X$, $Ax\subseteq X^*$,
and let
$\gra A := \menge{(x,x^*)\in X\times X^*}{x^*\in Ax}$ be
the \emph{graph} of $A$, and  $\dom A:= \menge{x\in X}{Ax\neq\varnothing}$ be the
\emph{domain} of $A$.
Recall that $A$ is  \emph{monotone} if
\begin{equation*}
\scal{x-y}{x^*-y^*}\geq 0,\quad \forall (x,x^*)\in \gra A\;\forall (y,y^*)\in\gra A.
\end{equation*}
We say $A$ is \emph{maximally monotone} if $A$ is monotone and $A$ has no proper monotone extension
(in the sense of graph inclusion).
Let $A:X\rightrightarrows X^*$ be monotone and $(x,x^*)\in X\times X^*$.
 We say $(x,x^*)$ is \emph{monotonically related to}
$\gra A$ if
\begin{align*}
\langle x-y,x^*-y^*\rangle\geq0,\quad \forall (y,y^*)\in\gra A.\end{align*}
Let $A:X\rightrightarrows X^*$ be maximally monotone. We say $A$ is
\emph{of type (FPV)} \cite{FitzPh, VV5} if  for every open convex set $U\subseteq X$ such that
$U\cap \dom A\neq\varnothing$, the implication
\begin{equation*}
x\in U\,\text{and}\,(x,x^*)\,\text{is monotonically related to $\gra A\cap (U\times X^*)$}
\Longrightarrow (x,x^*)\in\gra A
\end{equation*}
holds.

 Monotone operators have proven  important
 in modern Optimization and Analysis; see, e.g., the books
\cite{BC2011,BorVan,BurIus,ButIus,ph,Si,Si2,RockWets,Zalinescu,Zeidler2A,Zeidler2B}
and the references therein. We adopt standard notation used in these
books. Given a subset $C$ of $X$, $\inte C$ is the
\emph{interior} of $C$, $\overline{C}$ is the
\emph{norm closure} of $C$, and $\conv{C}$ is the \emph{convex hull} of $C$.
The \emph{indicator function} of $C$, written as $\iota_C$, is defined
at $x\in X$ by
\begin{align*}
\iota_C (x):=\begin{cases}0,\,&\text{if $x\in C$;}\\
+\infty,\,&\text{otherwise}.\end{cases}\end{align*}

If $C,D\subseteq X$, we set $C-D:=\{x-y\mid x\in C, y\in D\}$.
  For every $x\in X$, the \emph{normal cone} operator of $C$ at $x$
is defined by $N_C(x):= \menge{x^*\in
X^*}{\sup_{c\in C}\scal{c-x}{x^*}\leq 0}$, if $x\in C$; and $N_C(x):=\varnothing$,
if $x\notin C$.
For $x,y\in X$, we set $\left[x,y\right]:=\{tx+(1-t)y\mid 0\leq t\leq 1\}$.

 Given $f\colon X\to \RX$, we set
$\dom f:= f^{-1}(\RR)$.
 We say $f$ is \emph{proper} if $\dom f\neq\varnothing$.
 Let $f$ be proper.
Then
   $\partial f\colon X\To X^*\colon
   x\mapsto \menge{x^*\in X^*}{(\forall y\in
X)\; \scal{y-x}{x^*} + f(x)\leq f(y)}$ is the \emph{subdifferential
operator} of $f$. Thus $N_C=\partial\iota_C$.
 We also set $P_X: X\times X^*\rightarrow X\colon
(x,x^*)\mapsto x$. The \emph{open unit ball} in $X$ is
denoted by $U_X:= \menge{x\in X}{\|x\|< 1}$, the \emph{closed unit
ball} in $X$ is denoted by $B_X:= \menge{x\in X}{\|x\|\leq 1}$, and
$\NN:=\{1,2,3,\ldots\}$. We denote by $\longrightarrow$ and
$\weakstarly$ the norm convergence and weak$^*$ convergence of
nets,  respectively.

Let $A$ and $B$ be maximally monotone operators from $X$ to
$X^*$.
Clearly, the \emph{sum operator} $A+B\colon X\To X^*\colon x\mapsto
Ax+Bx: = \menge{a^*+b^*}{a^*\in Ax\;\text{and}\;b^*\in Bx}$
is monotone.
Rockafellar established the following significant result in 1970.
\begin{theorem}[Rockafellar's sum theorem]
\emph{(See  \cite[Theorem~1]{Rock70} or \cite{BorVan}.)} Suppose
that $X$ is reflexive. Let $A, B: X\rightrightarrows  X^*$ be
maximally monotone. Assume that $A$ and  $B$  satisfy the classical
\emph{constraint qualification}:\[\dom A \cap\intdom B\neq
\varnothing.\] Then $A+B$ is maximally monotone.
\end{theorem}
The generalization of Rockafellar's sum theorem in the setting of a reflexive space can be found
in \cite{AttRiaThe,Si2, SiZ,BorVan,FABVY}.

The most famous open problem in Monotone Operator Theory concerns  the
maximal monotonicity of the sum of two maximally monotone operators satisfying Rockafellar's constraint qualification
in general Banach spaces; this is called the ``sum problem''. Some
recent developments on the sum problem can be found in  Simons'
monograph \cite{Si2} and \cite{Bor1,Bor2,Bor3,BorVan,BY4FV,BY3, BY2, ZalVoi,
Voi1,
MarSva5,VV2,BWY4,BWY9, Yao3,Yao2,YaoPhD}, and also see \cite{AtBrezis} for the subdifferential operators.

In this paper, we focus on the  case when
 $A, B$ are maximally monotone
 with $\dom A\cap\inte\dom B\neq\varnothing$, and $A+N_{\overline{\dom B}}$ is of type (FPV) (see Theorem~\ref{TePGV:1}).

Corollary~\ref{CorPbA:1} provides an affirmative answer to the following problem
posed by Borwein and Yao in \cite[Open problem~4.5]{BY3}.
\begin{quote}
Let $f:X\rightarrow \RX$ be a proper lower semicontinuous convex function,
and let $B:X\rightrightarrows X^*$ be maximally monotone with $\dom \partial f\cap\inte\dom B\neq\varnothing$.
  Is
$\partial f +B$ necessarily maximally monotone?
\end{quote}

The remainder of this paper is organized as follows. In
Section~\ref{s:aux}, we collect auxiliary results for future
reference and for the reader's convenience. In Section~\ref{s:main}, our main
result (Theorem~\ref{TePGV:1}) is presented.  We also show that Problem~\ref{OPRKM:1} is equivalent to the sum problem.

\section{Auxiliary Results}
\label{s:aux}

We first introduce the well known
  Banach-Alaoglu Theorem and  the two of Rockafellar's results.
\begin{fact}[The Banach-Alaoglu Theorem]\label{BaAlo}
\emph{(See \cite[Theorem~3.15]{Rudin} or
\cite[Theorem~2.6.18]{Megg}.)}
The closed unit ball in $X^*$, $B_{X^*}$, is weakly$^*$ compact.
\end{fact}

\begin{fact}[Rockafellar] \label{f:F4}
\emph{(See {\cite[Theorem~3]{Rock66}},
{\cite[Theorem~18.1]{Si2}}, or
{\cite[Theorem~2.8.7(iii)]{Zalinescu}}.)}
Let $f,g: X\rightarrow\RX$ be proper convex functions.
Assume that there exists a point $x_0\in\dom f \cap \dom g$
such that $g$ is continuous at $x_0$.
Then  $\partial (f+g)=\partial f+\partial g$.
\end{fact}

\begin{fact}[Rockafellar]
\emph{(See \cite[Theorem~1]{Rock69} or
\cite[Theorem~27.1 and Theorem~27.3]{Si2}.)}
\label{f:refer02c}
Let $A:X\To X^*$ be  maximally monotone
 with $\inte\dom A\neq\varnothing$. Then
$\inte\dom A=\inte\overline{\dom A}$
and  $\inte\dom A$ and $\overline{\dom A}$ are both convex.
\end{fact}

The Fitzpatrick function defined below is an important tool in Monotone
Operator Theory.
\begin{fact}[Fitzpatrick]
\emph{(See {\cite[Corollary~3.9]{Fitz88}}.)}
\label{f:Fitz}
Let $A\colon X\To X^*$ be  monotone,  and set
\begin{equation*}
F_A\colon X\times X^*\to\RX\colon
(x,x^*)\mapsto \sup_{(a,a^*)\in\gra A}
\big(\scal{x}{a^*}+\scal{a}{x^*}-\scal{a}{a^*}\big),
\end{equation*}
 the \emph{Fitzpatrick function} associated with $A$.
Suppose also $A$ is maximally monotone. Then for every $(x,x^*)\in X\times X^*$, the inequality
$\scal{x}{x^*}\leq F_A(x,x^*)$ is true,
and the equality holds if and only if $(x,x^*)\in\gra A$.
\end{fact}

The next result is the key to our arguments.

\begin{fact}
\emph{(See \cite[Theorem~3.4 and Corollary~5.6]{Voi1},
or \cite[Theorem~24.1(b)]{Si2}.)}
\label{f:referee1}
Let $A, B:X\To X^*$ be maximally monotone operators. Assume
$\bigcup_{\lambda>0} \lambda\left[P_X(\dom F_A)-P_X(\dom F_B)\right]$
is a closed subspace.
If
\begin{equation*}
F_{A+B}\geq\langle \cdot,\,\cdot\rangle\;\text{on \; $X\times X^*$},
\end{equation*}
then $A+B$ is maximally monotone.
\end{fact}
Applying Fact~\ref{CoHull},
we can avoid computing the domain of the Fitzpatrick functions
in Fact~\ref{f:referee1} (see Corollary~\ref{VoiSimn:1} below).

\begin{fact}
\emph{(See \cite[Theorem~3.6]{BY2} or \cite{BY3}.)}
\label{CoHull}
Let $A:X\To X^*$ be a maximally monotone operator. Then
\begin{align*}
\overline{\conv\left[\dom A\right]}=\overline{P_X\left[\dom F_A\right]}.
\end{align*}
\end{fact}

\begin{lemma}\label{NonL:1}
Let $A, B\colon X\To X^*$ be  maximally monotone,
 and suppose that
$\bigcup_{\lambda>0}\lambda\left[\dom A-\dom B\right]$
 is a closed convex subset of $X$.
Then
\begin{align*}\bigcup_{\lambda>0}\lambda\left[\dom A-\dom B\right]=
\bigcup_{\lambda>0}\lambda\left[P_{X}\dom F_A-P_{X}\dom F_B\right].\end{align*}
\end{lemma}
\allowdisplaybreaks
\begin{proof}
By Fact~\ref{f:Fitz} and Fact~\ref{CoHull}, we have
\begin{align*}
&\bigcup_{\lambda>0} \lambda\left[\dom A-\dom B\right]\subseteq
\bigcup_{\lambda>0} \lambda\left[P_{X}\dom F_A-P_{X}\dom F_B\right]
\subseteq\bigcup_{\lambda>0} \lambda\left[\overline{\conv \dom A}-\overline{\conv \dom B}\right]\\
&\subseteq\bigcup_{\lambda>0} \lambda\left[\overline{\conv\dom A-\conv\dom B}\right]=\bigcup_{\lambda>0} \lambda\left[\overline{\conv\left[\dom A-\dom B\right]}\right]\subseteq
\overline{\bigcup_{\lambda>0} \lambda\conv\left[\dom A-\dom B\right]}\\
&=\bigcup_{\lambda>0} \lambda\left[\dom A-\dom B\right]\quad \text{(by the  assumption)}.
\end{align*}
Hence $\bigcup_{\lambda>0}\lambda\left[\dom A-\dom B\right]=
\bigcup_{\lambda>0}\lambda\left[P_{X}\dom F_A-P_{X}\dom F_B\right]$.
\end{proof}

\begin{corollary}
\label{VoiSimn:1}
Let $A, B:X\To X^*$ be maximally monotone operators. Assume that
$\bigcup_{\lambda>0} \lambda\left[\dom A-\dom B\right]$
is a closed subspace.
If
\begin{equation*}
F_{A+B}\geq\langle \cdot,\,\cdot\rangle\;\text{on \; $X\times X^*$},
\end{equation*}
then $A+B$ is maximally monotone.
\end{corollary}

\begin{proof}
Apply Fact~\ref{f:referee1} and Lemma~\ref{NonL:1} directly.
\end{proof}

Now we cite some results on operators of type (FPV).
\begin{fact}[Fitzpatrick-Phelps and Verona-Verona]
\emph{(See \cite[Corollary~3.4]{FitzPh}, \cite[Theorem~3]{VV1} or \cite[Theorem~48.4(d)]{Si2}.)}
\label{f:referee0d}\index{subdifferential operator}
Let $f:X\rightarrow\RX$ be proper, lower semicontinuous and convex.
Then $\partial f$ is of type (FPV).
\end{fact}

\begin{fact}[Simons]
\emph{(See \cite[Theorem~44.2]{Si2}.)}
\label{SDMn:pv}
Let $A:X\To X^*$ be  of type (FPV). Then
\begin{align*}
\overline{\dom A}=\overline{\conv \big(\dom A\big)}=\overline{P_X\big(\dom F_A\big)}.
\end{align*}
\end{fact}

The following result presents a sufficient condition for
a maximally monotone operator to be of type (FPV).
\begin{fact}[Simons and Verona-Verona]
\emph{(See \cite[Theorem~44.1]{Si2}, \cite{VV1} or \cite{Bor2}.)}
\label{f:refer02a}
Let $A:X\To X^*$ be maximally monotone. Suppose that
for every closed convex subset $C$ of $X$
with $\dom A \cap \inte C\neq \varnothing$, the operator
$A+N_C$ is maximally monotone.
Then $A$ is of type  (FPV).
\end{fact}

\begin{fact}[Boundedness below]\emph{(See  \cite[Fact~4.1]{BY1}.)}\label{extlem}
Let $A:X\rightrightarrows X^*$ be monotone and $x\in\inte\dom A$.
 Then there exist $\delta>0$ and  $M>0$ such that
 $x+\delta B_X\subseteq\dom A$ and
 $\sup_{a\in x+\delta B_X}\|Aa\|\leq M$.
Assume that $(z,z^*)$ is monotonically related to $\gra A$. Then
\begin{align*}
\langle z-x, z^*\rangle
\geq \delta\|z^*\|-(\|z-x\|+\delta) M.
\end{align*}
\end{fact}

We need the following bunch of useful tools from \cite{BY4FV}.
\begin{fact}\emph{(See  \cite[Proposition~3.1]{BY4FV}.)}\label{FProCVS}
Let $A:X\To X^*$ be   of type (FPV), and let
$B:X\rightrightarrows X^*$ be maximally monotone. Suppose
that $\dom A \cap \inte \dom B\neq\varnothing$.
Let $(z,z^*)\in X\times X^*$ with $z\in\overline{\dom B}$.
Then
\begin{align*}F_{A+B}(z,z^*)\geq\langle z,z^*\rangle.
\end{align*}
\allowdisplaybreaks
\end{fact}

\begin{fact}\emph{(See  \cite[Lemma~2.10]{BY4FV}.)}\label{LeWExc:1}
Let $A:X\To X^*$ be   monotone, and let
$B:X\rightrightarrows X^*$ be maximally monotone. Let $(z, z^*)\in X\times X^*$. Suppose
 $x_0\in\dom A \cap \inte \dom B$ and that there exists a sequence $(a_n, a^*_n)_{n\in\NN}$ in $\gra A\cap\Big(\dom B\times X^*\Big)$ such that $(a_n)_{n\in\NN}$ converges to a point in
$\left[x_0,z\right[$, and
\begin{align*}
\langle z-a_n, a^*_n\rangle\longrightarrow+\infty.
\end{align*}
Then
 $F_{A+B}(z,z^*)=+\infty$.
\end{fact}

\begin{fact}\emph{(See  \cite[Lemma~2.12]{BY4FV}.)}\label{LeWExc:3}
Let $A:X\To X^*$ be   of type (FPV). Suppose $x_0\in\dom A$ but
  that
$z\notin\overline{\dom A}$. Then there exists a sequence
$(a_n, a^*_n)_{n\in\NN}$ in $\gra A$ such that $(a_n)_{n\in\NN}$ converges to a point in
$\left[x_0,z\right[$ and
\begin{align*}
\langle z-a_n, a^*_n\rangle\longrightarrow+\infty.
\end{align*}
\end{fact}

The proof of Fact~\ref{FCTV:1} and Fact~\ref{FCTV:2} is mainly extracted from the part of the proof of \cite[Proposition~3.2]{BY4FV}.

\begin{fact}\label{FCTV:1}
Let $A:X\To X^*$ be maximally monotone and $z\in\overline{\dom A}\backslash\dom A$.
Then for every  sequence $(z_n)_{n\in\NN}$ in $\dom A$ such that $z_n\longrightarrow
z$, we have $\lim_{n\rightarrow\infty}\inf\|A(z_n)\|=+\infty$.
\end{fact}
\begin{proof} Suppose to the contrary that
there exists a sequence $z^*_{n_k}\in A(z_{n_k})$  and $L>0$ such that $\sup_{k\in\NN}\|z^*_{n_k}\|\leq L$.
By Fact~\ref{BaAlo},
 there  exists a weak* convergent subnet, $(z^*_{\beta})_{\beta\in J}$ of
$(z^*_{n_k})_{k\in\NN}$ such that $z^*_{\beta}\weakstarly z^*_{\infty}\in X^*$.
\cite[Fact~3.5]{BY1} or \cite[Section~2, page~539]{BFG} shows that $(z, z^*_{\infty})\in\gra A$,
which contradicts our assumption that $z\notin \dom A$.
Hence we have our result holds.
\end{proof}

\begin{fact}\label{FCTV:2}
Let $A, B:X\To X^*$ be  monotone. Let $(z, z^*)\in X\times X^*$. Suppose
that $x_0\in\dom A \cap \inte \dom B$ and that there exist a sequence $(a_n, a^*_n)_{n\in\NN}$ in $\gra A\cap\big(\dom B\times X^*)$ and a sequence $(K_n)_{n\in\NN}$
 in $\RR$ such that $(a_n)_{n\in\NN}$ converges to a point in
$\left[x_0,z\right[$, and that
\begin{align}
\langle z- a_n, a^*_n\rangle\geq K_n.\label{FCTV:2:e1}
\end{align}
Assume that there exists a sequence $b^*_n\in Ba_n$ such that
$\frac{K_n}{\|b^*_n\|}\longrightarrow 0$ and $\|b^*_n\|\longrightarrow
+\infty$.
Then $F_{A+B}(z,z^*)=+\infty$.
\end{fact}

\begin{proof}
By the assumption, there exists $0\leq\delta<1$ such that
\begin{align}
a_n\longrightarrow x_0+\delta(z-x_0).\label{FCTV:2:e2}
\end{align}
Suppose to the contrary that
\begin{align}
F_{A+B}(z,z^*)<+\infty.\label{FCTV:2:e3}
\end{align}
By Fact~\ref{BaAlo},
 there  exists a weak* convergent subnet, $(\frac{b^*_i}{\|b^*_i\|})_{i\in I}$ of
$\frac{b^*_n}{\|b^*_n\|}$ such that
\begin{align}
\frac{b^*_i}{\|b^*_i\|}\weakstarly b^*_{\infty}\in X^*.\label{PCSM:c4}
\end{align}
By \eqref{FCTV:2:e1}, we have
\begin{align*}
K_n+\Big\langle z-a_n, b^*_n\Big\rangle+\Big\langle z^*, a_n\Big\rangle&\leq
\Big\langle z-a_n, a_n^*\Big\rangle+
\Big\langle z-a_n, b^*_n\Big\rangle+\Big\langle z^*, a_n\Big\rangle\\
&\leq F_{A+B}(z,z^*)
\end{align*}
Thus
\begin{align}
\frac{K_n}{\|b^*_n\|}
+\Big\langle z-a_n, \frac{b^*_n}{\|b^*_n\|}\Big\rangle+\frac{1}{\|b^*_n\|}\Big\langle z^*, a_n\Big\rangle
&\leq \frac{F_{A+B}(z,z^*)}{\|b^*_n\|}.\label{PCSM:c3}
\end{align}

By the assumption that $\frac{K_n}{\|b^*_n\|}\longrightarrow 0$ and $\|b^*_n\|\longrightarrow
+\infty$, \eqref{FCTV:2:e2}, \eqref{FCTV:2:e3} and \eqref{PCSM:c4}, we  take the limit along the subnet in \eqref{PCSM:c3} to obtain
\begin{align*}
\Big\langle z-x_0-\delta(z-x_0), b^*_{\infty}\Big\rangle\leq0.
\end{align*}
Since $\delta<1$,
\begin{align}
\Big\langle z-x_0, b^*_{\infty}\Big\rangle\leq0.\label{PCSM:c6}
\end{align}

On the other hand, since $x_0\in\inte\dom B$
and $(a_n, b^*_n)\in\gra B$, Fact~\ref{extlem} implies that there exist
$\eta>0$ and $M>0$ such that
\begin{align*}\langle a_n-x_0, b_n^*\rangle
\geq\eta\|b^*_n\|-(\|a_n-x_0\|
+\eta)M.
\end{align*}
Thus
\begin{align*}
\langle a_n-x_0, \frac{b_n^*}{\|b^*_n\|}\rangle
\geq\eta-\frac{(\|a_n-x_0\|
+\eta)M}{\|b^*_n\|}.
\end{align*}
Since $\|b^*_n\|\longrightarrow +\infty$, by \eqref{FCTV:2:e2} and \eqref{PCSM:c4}, we
 take the limit along the subnet in the above inequality to obtain
\begin{align*}
\Big\langle x_0+\delta(z-x_0)-x_0, b^*_{\infty}\Big\rangle\geq\eta.
\end{align*}
Hence
\begin{align*}
\Big\langle z-x_0, b^*_{\infty}\Big\rangle\geq\frac{\eta}{\delta}>0,
\end{align*}
which contradicts \eqref{PCSM:c6}.  Hence $F_{A+B}(z,z^*)=
+\infty$.
\end{proof}

\section{Our main result}
\label{s:main}

The following result is the key technical tool for our main result (: Theorem~\ref{TePGV:1}).  The proof of Proposition~\ref{ProCVS:P1}
  follows in part that of \cite[Proposition~3.2]{BY4FV}.
\begin{proposition}\label{ProCVS:P1}
Let $A:X\To X^*$ be   of type (FPV), and let
$B:X\rightrightarrows X^*$ be maximally monotone. Suppose
 $x_0\in\dom A\cap \inte \dom B$ and  $(z,z^*)\in X\times X^*$. Assume
  that there  exist a sequence $(a_n)_{n\in\NN}$ in $\dom A\cap\left[\overline{\dom B}\backslash\dom B\right]$ and $\delta\in\left[0,1\right]$ such that $a_n\longrightarrow\delta z+(1-\delta) x_0$.
Then
$F_{A+B}(z,z^*)\geq\langle z,z^*\rangle$.

\end{proposition}
\begin{proof}
Suppose to the contrary that
\begin{align}
F_{A+B}(z,z^*)<\langle z,z^*\rangle.\label{EProF2:e1}
\end{align}
By the assumption, we have $\delta z+(1-\delta) x_0\in\overline{\dom B}$.
Since $a_n\notin\dom B$ and $x_0\in\inte\dom B$, Fact~\ref{FProCVS} and \eqref{EProF2:e1} imply that
\begin{align}
0<\delta<1 \quad\text{and }\quad \delta z+(1-\delta) x_0\neq x_0.\label{PCSMaL:ec1}
\end{align}
We set
\begin{align}
y_0:=\delta z+(1-\delta) x_0.\label{PCSMaL:ea1}
\end{align}
Since $a_n\in\dom A$, we let
\begin{align}
(a_n, a^*_n)\in\gra A,\quad\forall n\in\NN.
\end{align}

Since $x_0\in\dom A\cap \inte \dom B$, there exist $x^*_0, y^*_0\in X^*$ such that
$(x_0, x^*_0)\in\gra A$ and $(x_0, y^*_0)\in\gra B$. By $x_0\in\inte\dom B$, there exists $0<\rho_0\leq \|y_0-x_0\|$ by \eqref{PCSMaL:ec1} such that
\begin{align}x_0+\rho_0 U_X\subseteq\dom B.\label{PCSM:cc1}
\end{align}
Now we show that  there exists $\delta\leq t_n\in\left[1-\frac{1}{n},1\right[$ such that
 that
\begin{align}
H_n\subseteq\dom B~\mbox{and~}
\inf\big\|B\big(H_n\big)\big\|\geq 4K^2_0 (\|a^*_n\|+1)n,
\label{PCSM:c1}\end{align}
where \begin{align}H_n:&=t_na_n+(1-t_n) x_0+(1-t_n) \rho_0U_X\nonumber\\
K_0:&=\max\Big\{3\| z\|+2+3|x_0\|, \,
 \frac{1}{\delta}\big(\frac{2\|y_0-x_0\|}{\rho_0}+1\big)\big(\|x^*_0\|+1\big)\Big\}.
 \label{PCSMaL:ca1}
\end{align}
For every $s\in \left]0,1\right[$,  since $a_n\in\overline{\dom B}$,
 \eqref{PCSM:cc1} and Fact~\ref{f:refer02c} imply that
\begin{align*}sa_n+(1-s) x_0+(1-s) \rho_0 U_X=sa_n+(1-s)\left[x_0+ \rho_0 U_X\right]\subseteq \overline{\dom B}.
\end{align*}
By Fact~\ref{f:refer02c} again,
$sa_n+(1-s) x_0+(1-s) \rho_0 U_X\subseteq \inte\overline{\dom B}=\inte\dom B$.

It directly follows from Fact~\ref{FCTV:1} and $a_n\in\overline{\dom B}\backslash\dom B$ that the second part of \eqref{PCSM:c1} holds.

Set
\begin{align}
r_n:= \frac{\frac{1}{2}(1-t_n)\rho_0}{t_n\|y_0-a_n\|+(1-t_n)\|y_0-x_0\|}
.\label{PCSMaL:ed1}
\end{align}
Since $\rho_0\leq \|y_0-x_0\|$, we have $r_n\leq\frac{1}{2}$.
Now we show that
\begin{align}
v_n:&=r_n y_0+ (1-r_n)\left[t_na_n+(1-t_n) x_0\right]\nonumber\\
&=r_n\delta z+(1-r_n)t_n a_n+s_nx_0\in H_n,\label{PCSMaL:e5}
\end{align}
where $s_n:=\left[1-t_n+r_n(t_n-\delta)\right]$.

Indeed, we have
\begin{align*}
&\Big\|v_n-t_na_n-(1-t_n) x_0\Big\|=
\Big\|r_n y_0+ (1-r_n)\left[t_na_n+(1-t_n) x_0\right]-t_na_n-(1-t_n) x_0\Big\|\\
&=\Big\|r_n y_0-r_n\left[t_na_n+(1-t_n) x_0\right]\Big\|
=r_n\Big\|t_ny_0+(1-t_n)y_0-\left[t_na_n+(1-t_n) x_0\right]\Big\|\\
&=r_n\Big\|t_n(y_0-a_n)+(1-t_n)(y_0-x_0)\Big\|\leq
r_n\Big(t_n\|y_0-a_n\|+(1-t_n)\|y_0-x_0\|\Big)\\
&=\frac{1}{2}(1-t_n)\rho_0\quad\text{(by \eqref{PCSMaL:ed1})}.
\end{align*}
Hence $v_n\in H_n$ and thus \eqref{PCSMaL:e5} holds by \eqref{PCSMaL:ea1}.

Since $a_n\longrightarrow y_0$ and $v_n\in H_n$ by \eqref{PCSMaL:e5},
$v_n\longrightarrow y_0$.  Then we can and do suppose that
\begin{align}
\|v_n\|\leq\|y_0\|+1\leq\|z\|+\|x_0\|+1,\quad\forall n\in\NN\quad\text{(by \eqref{PCSMaL:ea1})}.\label{PCSMaL:cc2}
\end{align}
Since $a_n\longrightarrow y_0$ and $\|y_0-x_0\|>0$ by \eqref{PCSMaL:ec1}, we can suppose that
\begin{align*}
\|y_0-a_n\|\leq\|y_0-x_0\|,\quad\forall n\in\NN.
\end{align*}
Then by \eqref{PCSMaL:ed1},
\begin{align}
\frac{1-t_n}{r_n}\leq\frac{2\|y_0-x_0\|}{\rho_0},\quad\forall n\in\NN. \label{PCSMaL:c1}
\end{align}
Since $s_n=\left[1-t_n+r_n(t_n-\delta)\right]$, by \eqref{PCSMaL:c1} and $\delta\leq t_n<1$, we have
\begin{align}
\frac{s_n}{r_n}=\frac{1-t_n}{r_n}+t_n-\delta\leq \frac{2\|y_0-x_0\|}{\rho_0}+1,\quad\forall n\in\NN. \label{PCSMaL:c2}
\end{align}

Now  we show there exists $(\widetilde{a_n}, \widetilde{a_n}^*)_{n\in\NN}$
in $\gra A\cap (H_n\times X^*)$ such that
\begin{align}
\big\langle z-\widetilde{a_n},\widetilde{a_n}^*\big\rangle\geq-4K^2_0 (\|a^*_n\|+1).\label{EProF2:e2}
\end{align}
We consider two cases.

\emph{Case 1}: $(v_n,(2-t_n)a^*_n)\in\gra A$.

Set $(\widetilde{a_n}, \widetilde{a_n}^*):=(v_n,(2-t_n)a^*_n)$. Then we have
\begin{align}
&\langle z- \widetilde{a_n}, \widetilde{a_n}^*\rangle
=\langle z- v_n, (2-t_n){a_n}^*\rangle\geq-2\|z- v_n\|\cdot\|a^*_n\|\nonumber\\
&\geq-2\big(2\|z\|+\|x_0\|+1\big)\cdot\|a^*_n\|\geq-4K^2_0 (\|a^*_n\|+1)\quad\text{(by \eqref{PCSMaL:cc2} and
\eqref{PCSMaL:ca1})}.
\end{align}
Hence \eqref{EProF2:e2} holds since $v_n\in H_n$ by \eqref{PCSMaL:e5}.

\emph{Case 2}: $(v_n,(2-t_n)a^*_n)\notin\gra A$.

By Fact~\ref{SDMn:pv} and the assumption that $\{a_n, y_0, x_0\}\subseteq\overline{\dom A}$,  \eqref{PCSMaL:e5} shows that $v_n\in\overline{\dom A}$. Thus $H_n\cap\dom A\neq\varnothing$ by \eqref{PCSMaL:e5} again.
 Since $\Big(v_n,(2-t_n)a^*_n\Big)\notin\gra A$, $v_n\in H_n$ by \eqref{PCSMaL:e5},
  and $A$ is of type (FPV), there exists $(\widetilde{a_n}, \widetilde{a_n}^*)\in\gra A\cap (H_n\times X^*)$ such that
  \begin{align*}
  \Big\langle v_n-\widetilde{a_n}, \widetilde{a_n}^*-(2-t_n)a^*_n\Big\rangle>0.
  \end{align*}
 Thus by \eqref{PCSMaL:e5}, we have
 \begin{align}
 &\Big\langle v_n-\widetilde{a_n}, \widetilde{a_n}^*-(2-t_n)a^*_n\Big\rangle>0\nonumber\\
 &\Longrightarrow \Big\langle r_n\delta z+(1-r_n)t_n a_n+s_nx_0 -\widetilde{a_n}, \widetilde{a_n}^*-a^*_n-(1-t_n)a^*_n\Big\rangle>0\nonumber\\
 &\Longrightarrow \Big\langle r_n\delta z+(1-r_n)t_n a_n+s_nx_0 -\widetilde{a_n}, \widetilde{a_n}^*-a^*_n\Big\rangle\nonumber\\
 &\quad>\Big\langle r_n\delta z+(1-r_n)t_n a_n+s_nx_0 -\widetilde{a_n}, (1-t_n)a^*_n\Big\rangle\nonumber\\
 &\quad\geq- (1-t_n)\Big(\| z\|+\|a_n\|+\|x_0\| +\|\widetilde{a_n}\|\Big) \|a^*_n\|\label{PCSMaL:e12}
 \end{align}
Note that $\widetilde{a_n}=r_n\delta\widetilde{a_n}+(1-r_n)t_n\widetilde{a_n}+s_n\widetilde{a_n}$.
Thus \eqref{PCSMaL:e12} implies that
 \begin{align}
 &\Big\langle r_n\delta(z-\widetilde{a_n})+(1-r_n)t_n (a_n-\widetilde{a_n})+s_n(x_0-\widetilde{a_n}), \widetilde{a_n}^*-a^*_n\Big\rangle\nonumber\\
 &\quad>- (1-t_n)\Big(\| z\|+\|a_n\|+\|x_0\| +\|\widetilde{a_n}\|\Big) \|a^*_n\|\nonumber\\
 &\Longrightarrow
 \Big\langle r_n\delta(z-\widetilde{a_n})+s_n(x_0-\widetilde{a_n}), \widetilde{a_n}^*-a^*_n\Big\rangle\nonumber\\
 &\quad\geq(1-r_n)t_n\Big\langle a_n-\widetilde{a_n}, a^*_n-\widetilde{a_n}^*\Big\rangle- (1-t_n)\Big(\| z\|+\|a_n\|+\|x_0\| +\|\widetilde{a_n}\|\Big) \|a^*_n\|\nonumber\\
 &\quad\geq- (1-t_n)\Big(\| z\|+\|a_n\|+\|x_0\| +\|\widetilde{a_n}\|\Big) \|a^*_n\|
 \quad\text{(by the monotonicity of $A$)}\nonumber\\
 &\Longrightarrow
 \Big\langle r_n\delta(z-\widetilde{a_n})+s_n(x_0-\widetilde{a_n}), \widetilde{a_n}^*\Big\rangle\nonumber\\
 &\quad> \Big\langle r_n\delta(z-\widetilde{a_n})+s_n(x_0-\widetilde{a_n}), a^*_n\Big\rangle- (1-t_n)\Big(\| z\|+\|a_n\|+\|x_0\| +\|\widetilde{a_n}\|\Big) \|a^*_n\|\nonumber\\
 &\Longrightarrow
 r_n\delta\Big\langle z-\widetilde{a_n}, \widetilde{a_n}^*\Big\rangle
> s_n\Big\langle \widetilde{a_n}-x_0, \widetilde{a_n}^*\Big\rangle+\Big\langle r_n\delta(z-\widetilde{a_n})+s_n(x_0-\widetilde{a_n}), a^*_n\Big\rangle\nonumber\\
 &\quad- (1-t_n)\Big(\| z\|+\|a_n\|+\|x_0\| +\|\widetilde{a_n}\|\Big) \|a^*_n\|\label{PCSMaL:e13}
 \end{align}
 Since $\{(x_0, x^*_0), (\widetilde{a_n}, \widetilde{a_n}^*)\}\subseteq\gra A$,
 we have $\langle \widetilde{a_n}-x_0, \widetilde{a_n}^*\rangle\geq
 \langle \widetilde{a_n}-x_0, x_0^*\rangle$ by the monotonicity of $A$. Thus, by \eqref{PCSMaL:e13},
 \begin{align}
 &r_n\delta\Big\langle z-\widetilde{a_n}, \widetilde{a_n}^*\Big\rangle
> s_n\Big\langle \widetilde{a_n}-x_0, x_0^*\Big\rangle+\Big\langle r_n\delta(z-\widetilde{a_n})+s_n(x_0-\widetilde{a_n}), a^*_n\Big\rangle\nonumber\\
 &\quad- (1-t_n)\Big(\| z\|+\|a_n\|+\|x_0\| +\|\widetilde{a_n}\|\Big) \|a^*_n\|\nonumber\\
&\geq -s_n\|\widetilde{a_n}-x_0\|\cdot\| x_0^*\|-\ r_n\|z-\widetilde{a_n}\|\cdot\|a^*_n\|-s_n\|x_0-\widetilde{a_n}\|\cdot \|a^*_n\|\nonumber\\
 &\quad- (1-t_n)\Big(\| z\|+\|a_n\|+\|x_0\| +\|\widetilde{a_n}\|\Big) \|a^*_n\|\nonumber
 \end{align}
 Hence
  \begin{align}
 &\Big\langle z-\widetilde{a_n}, \widetilde{a_n}^*\Big\rangle
> -\frac{s_n}{r_n\delta}\|\widetilde{a_n}-x_0\|\cdot\| x_0^*\|-\ \frac{1}{\delta}\|z-\widetilde{a_n}\|\cdot\|a^*_n\|-\frac{s_n}{r_n\delta}\|x_0-\widetilde{a_n}\|\cdot \|a^*_n\|\nonumber\\
 &\quad- \frac{1-t_n}{r_n\delta}\Big(\| z\|+\|a_n\|+\|x_0\| +\|\widetilde{a_n}\|\Big) \|a^*_n\|\nonumber
 \end{align}
 Then combining \eqref{PCSMaL:c1} and \eqref{PCSMaL:c2}, we have
 \begin{align}
 &\Big\langle z-\widetilde{a_n}, \widetilde{a_n}^*\Big\rangle
> -(\frac{2\|y_0-x_0\|}{\rho_0}+1)\frac{1}{\delta}\|\widetilde{a_n}-x_0\|\cdot\| x_0^*\|-\ \frac{1}{\delta}\|z-\widetilde{a_n}\|\cdot\|a^*_n\|
\nonumber\\
 &\quad-(\frac{2\|y_0-x_0\|}{\rho_0}+1)\frac{1}{\delta}\|x_0-\widetilde{a_n}\|\cdot \|a^*_n\|- \frac{2\|y_0-x_0\|}{\rho_0\delta}\Big(\| z\|+\|a_n\|+\|x_0\| +\|\widetilde{a_n}\|\Big) \|a^*_n\|\nonumber\\
 &\quad\geq
 -K_0\|\widetilde{a_n}-x_0\|-K_0\|z-\widetilde{a_n}\|\cdot\|a^*_n\|
-K_0\|x_0-\widetilde{a_n}\|\cdot \|a^*_n\|\nonumber\\
&\quad- K_0\Big(\| z\|+\|a_n\|+\|x_0\| +\|\widetilde{a_n}\|\Big) \|a^*_n\|\quad\text{(by \eqref{PCSMaL:ca1})}\label{PCSMaL:e14}
 \end{align}
 Since $\widetilde{a_n}\in H_n$, $t_n\longrightarrow 1^{-}$
 and $a_n\longrightarrow
 y_0$,
 \begin{equation}
 \widetilde{a_n}\longrightarrow y_0. \label{PCSMaL:e9}
 \end{equation}
 Then we can and do suppose that
 \begin{align}
 \max\big\{\|a_n\|,\|\widetilde{a_n}\|\big\}\leq \|y_0\|+1\leq\|x_0\|+\|z\|+1, \quad\forall n\in\NN.\quad\text{(by \eqref{PCSMaL:ea1})}\label{PCSMaL:e15}
 \end{align}
 Then by \eqref{PCSMaL:e15}, \eqref{PCSMaL:e14} and \eqref{PCSMaL:ca1}, we have
 \begin{align}
 &\Big\langle z-\widetilde{a_n}, \widetilde{a_n}^*\Big\rangle
>-K_0^2-K_0^2\|a^*_n\|
-K_0^2 \|a^*_n\|- K_0^2\|a^*_n\|\geq -4K^2_0 (\|a^*_n\|+1).
\end{align}
Hence \eqref{EProF2:e2} holds.

Combining the above two cases, we have \eqref{EProF2:e2} holds.

Since $\widetilde{a_n}\in H_n$, \eqref{PCSM:c1} implies that $\widetilde{a_n}\in\dom B$.  Then combining \eqref{PCSMaL:e9}, \eqref{PCSMaL:ec1},
\eqref{PCSM:c1} and \eqref{EProF2:e2},  Fact~\ref{FCTV:2} implies that
$F_{A+B}(z,z^*)=+\infty$, which contradicts \eqref{EProF2:e1}.

Hence $F_{A+B}(z,z^*)\geq\langle z,z^*\rangle$.
\end{proof}

Now we come to our main result.
\begin{theorem}[Main result]\label{TePGV:1}
Let $A, B:X\To X^*$ be  maximally monotone
with $\dom A\cap\inte\dom B\neq\varnothing$. Assume
that $A+N_{\overline{\dom B}}$ is of type (FPV). Then
$A+B$ is maximally monotone.
\end{theorem}
\begin{proof}
After translating the graphs if necessary, we can and do assume that
$0\in\dom A\cap\inte\dom B$ and that $(0,0)\in\gra A\cap\gra B$.
By Corollary~\ref{VoiSimn:1}, it suffices to show that
\begin{equation} \label{EOL:1}
F_{A+ B}(z,z^*)\geq \langle z,z^*\rangle,\quad \forall(z,z^*)\in X\times X^*.
\end{equation}
Take $(z,z^*)\in X\times X^*$.
%Then
%\begin{align}
%&F_{A+B}(z,z^*)\nonumber\\
%&=\sup_{\{x,x^*,y^*\}}\left[\langle x,z^*\rangle+\langle z-x,x^*\rangle
%+\langle z-x, y^*\rangle -\iota_{\gra A}(x,x^*)-\iota_{\gra
%B}(x,y^*)\right].\label{see:1}
%\end{align}
Suppose to the contrary that
\begin{align}
F_{A+B}(z,z^*)<\langle z,z^*\rangle.\label{FPCoS:e20}
\end{align}
%so that
%\begin{align}
%(z,z^*)\,\text{ is monotonically related to $\gra (A+B)$}.\label{SDFC:47}\end{align}

Since $B$ is maximally monotone, $B=B+N_{\overline{\dom B}}$.  Thus
\begin{align}
A+B= A+B+N_{\overline{\dom B}}= (A+N_{\overline{\dom B}})+B.\label{TePGVL:e1}
\end{align}
Since $A+N_{\overline{\dom B}}$ is of type (FPV) and $0\in\dom
\left[A+N_{\overline{\dom B}}\right]\cap\inte\dom B$,
Fact~\ref{FProCVS} and \eqref{FPCoS:e20} imply that
\begin{align}
z\notin\overline{\dom B}\quad\text{and then}\quad
z\notin\overline{\dom \left[A+N_{\overline{\dom B}}\right]}
\label{TePGVL:e2}
\end{align}
Then by Fact~\ref{LeWExc:3},
there exist a sequence
$(a_n, a^*_n)_{n\in\NN}$ in $\gra (A+N_{\overline{\dom B}})$
and $\delta\in\left[0,1\right[$ such that
\begin{align}
a_n\longrightarrow \delta z\quad\text{and}\quad
\langle z-a_n, a^*_n\rangle\longrightarrow+\infty.\label{TePGVL:e3}
\end{align}
Thus $a_n\in\dom \left[A+N_{\overline{\dom B}}\right]\cap\overline{\dom B},\, \forall n\in\NN$.

Now we consider two cases.

\emph{Case 1}:  There exists a subsequence of $(a_n)_{n\in\NN}$ in $\dom B$.

We can and do suppose that $a_n\in\dom B$ for every $n\in\NN$. Thus $(a_n)_{n\in\NN}$
is in  $ \dom \left[A+N_{\overline{\dom B}}\right]\cap\dom B$.

Combining Fact~\ref{LeWExc:1} and \eqref{TePGVL:e3},
\begin{align*}F_{A+B}(z,z^*)=F_{A+N_{\overline{\dom B}}+B}(z,z^*)
=+\infty,
\end{align*}
which contradicts \eqref{FPCoS:e20}.

\emph{Case 2}:  There exists $N_1\in\NN$ such that $a_n\notin\dom B,\, \forall n\geq N_1$.

Then we can and do suppose that $a_n\notin\dom B$ for every $n\in\NN$.
Thus, $a_n\in\dom \left[A+N_{\overline{\dom B}}\right]\cap\left[\overline{\dom B}\backslash\dom B\right]$.
By Proposition~\ref{ProCVS:P1} and \eqref{TePGVL:e3}, \begin{align*}F_{A+B}(z,z^*)=F_{A+N_{\overline{\dom B}}+B}(z,z^*)\geq\langle z,z^*\rangle,
\end{align*}
which contradicts \eqref{FPCoS:e20}.

Combing all the above cases, we have $F_{A+B}(z,z^*)\geq\langle z, z^*\rangle$ for all $(z,z^*)\in X\times X^*$.
Hence $A+B$ is  maximally monotone.
\end{proof}

\begin{remark}Theorem~\ref{TePGV:1} generalizes the main result in
\cite{Yao3} (see \cite[Theorem~3.4]{Yao3}).
\end{remark}

\begin{corollary}\label{CorPbA:1}
Let $f:X\rightarrow \RX$ be a proper lower semicontinuous convex function,
and let $B:X\rightrightarrows X^*$ be maximally monotone with $\dom \partial f\cap\inte\dom B\neq\varnothing$.
 Then
$\partial f +B$ is maximally monotone.
\end{corollary}
\begin{proof}
By Fact~\ref{f:refer02c} and Fact~\ref{f:F4} (or \cite[Theorem~1.1]{AtBrezis}), $\partial f+N_{\overline{\dom B}}
=\partial (f+\iota_{\overline{\dom B}})$.
Then Fact~\ref{f:referee0d} shows that $\partial f+N_{\overline{\dom B}}$ is
of type (FPV).  Applying Theorem~\ref{TePGV:1},
we have $\partial f +B$ is maximally monotone.
\end{proof}

\begin{remark}
Corollary~\ref{CorPbA:1} provides an affirmative answer to a problem posed by Borwein and Yao in \cite[Open problem~4.5]{BY3}.
\end{remark}

Given a set-valued operator $A:X\rightrightarrows X^*$,
we say $A$ is a \emph{linear relation} if $\gra A$ is a
linear subspace.
\begin{corollary}[Linear relation]\emph{(See \cite[Theorem~3.1]{BY3} or \cite[Corollary~4.5]{BY4FV}.)}
Let $A:X\To X^*$ be a maximally monotone linear relation, and let
$B: X\rightrightarrows X^*$ be maximally monotone. Suppose  that
$\dom A\cap\inte\dom B\neq\varnothing$.  Then $A+B$ is maximally
monotone.
\end{corollary}

\begin{proof}
Apply Fact~\ref{f:refer02c}, \cite[Corollary~3.3]{Yao2} and Theorem~\ref{TePGV:1} directly.
\end{proof}

\begin{corollary}[Convex domain]\emph{(See \cite[Corollary~4.3]{BY4FV}.)}
Let $A:X\To X^*$ be of type (FPV) with convex domain, and let
$B: X\rightrightarrows X^*$ be maximally monotone. Suppose  that
$\dom A\cap\inte\dom B\neq\varnothing$.  Then $A+B$ is maximally
monotone.
\end{corollary}
\begin{proof}
Apply Fact~\ref{f:refer02c}, \cite[Corollary~2.10]{Yao3} and Theorem~\ref{TePGV:1} directly.
\end{proof}

Applying Fact~\ref{f:refer02a} and Theorem~\ref{TePGV:1}, we can obtain that
   the sum problem is equivalent to the following problem:

\begin{problem}\label{OPRKM:1}
Let $A:X\To X^*$ be  maximally monotone, and $C$ be a nonempty closed and convex subset of $X$.
Assume that $\dom A\cap\inte C\neq\varnothing$. Is
$A+N_C$  necessarily maximally monotone?
\end{problem}
Clearly,   Problem~\ref{OPRKM:1} is a special case of the sum
problem.
However, if we would have an affirmative answer to Problem~\ref{OPRKM:1} for every maximally monotone operator $A$ and
every nonempty closed and convex set $C$ satisfying Rockafellar's constraint qualification:
$\dom A\cap\inte C\neq\varnothing$. Then Fact~\ref{f:refer02a} implies that
$A$ is of type (FPV), and then $\overline{\dom A}$ is convex by Fact~\ref{SDMn:pv}. Applying Fact~\ref{f:refer02a} again and using the technique similar to the proof of
\cite[Corollary~4.6]{BY4FV} (or \cite[Corollary~2.10]{Yao3}), we can obtain that
$A+N_C$ is of type (FPV).  Thus applying Theorem~\ref{TePGV:1}, we have
an affirmative answer to the sum problem.

\section*{Acknowledgment}
The author thanks Dr.~Anthony Lau for supporting his visit and
providing excellent working conditions.


\begin{thebibliography}{99}


\bibitem{FABVY} F.J.\ Arag\'{o}n Artacho, J.M.\ Borwein,
V.\ Mart\'{i}n-M\'{a}rquez, and L.\ Yao,``Applications of convex analysis within mathematics'', \emph{Mathematical Programming (Series B)}, in press;
\texttt{http://dx.doi.org/10.1007/s10107-013-0707-3}.


\bibitem{AtBrezis}
H.\ Attouch and H.\ Br\'{e}zis,
``Duality for the sum of convex functions in general Banach spaces",
\emph{Aspects of Mathematics and its Applications}, J. A.
Barroso, ed., Elsevier Science Publishers, pp.~125--133, 1986.



\bibitem{AttRiaThe}
H.\ Attouch, H.\ Riahi, and M.\ Thera, ``Somme ponctuelle d'operateurs maximaux monotones" [Pointwise sum of maximal monotone operators] Well-posedness and stability of variational problems. \emph{Serdica. Mathematical Journal}, vol.~22, pp.~165--190, 1996.

\bibitem{BC2011}
H.H.\ Bauschke and P.L.\ Combettes,
\emph{Convex Analysis and Monotone Operator Theory in Hilbert Spaces},
Springer-Verlag, 2011.





\bibitem{BWY4}
H.H.\ Bauschke, X.\ Wang, and L.\ Yao,
``An answer to S.\ Simons' question
on the maximal monotonicity of the sum of a
maximal monotone linear operator and a normal cone operator'',
\emph{Set-Valued and Variational Analysis}, vol.~17, pp.~195--201, 2009.




\bibitem{BWY9}
H.H.\ Bauschke, X.\ Wang, and L.\ Yao,
``On the maximal monotonicity of the sum  of a maximal monotone
linear relation and the subdifferential operator
of a sublinear function'', \emph{Proceedings of the Haifa Workshop on
 Optimization Theory and Related Topics.
Contemp. Math., Amer. Math. Soc., Providence, RI}, vol.~568, pp.~19--26, 2012.







\bibitem{Bor1}
J.M.\ Borwein,
``Maximal monotonicity via convex analysis'',
\emph{Journal of Convex Analysis}, vol.~13, pp.~561--586, 2006.

\bibitem{Bor2}
J.M.\ Borwein, ``Maximality of sums of two maximal monotone operators in general
Banach space'',
\emph{Proceedings of the
 American Mathematical Society}, vol.~135, pp.~3917--3924, 2007.

\bibitem{Bor3}J.M.\ Borwein, ``Fifty years of maximal monotonicity'',
\emph{Optimization Letters}, vol.~4, pp.~473--490, 2010.

\bibitem{BFG}J.M.\ Borwein, S.\ Fitzpatrick, and R.\ Girgensohn,
``Subdifferentials whose graphs are not norm $\times$ weak$^*$ closed'',
\emph{Canadian Mathematical Bulletin}, vol.~4, pp.~538--545, 2003.


\bibitem{BorVan}
J.M. Borwein and J. Vanderwerff, \emph{Convex Functions:
Constructions, Characterizations and Counterexamples}, Encyclopedia
of Mathematics and its Applications, \textbf{109}, Cambridge University
Press, 2010.


 \bibitem{BY1}  J.M.\ Borwein and L.\ Yao,``Structure theory for maximally monotone operators with points of continuity'', \emph{Journal of Optimization Theory and Applications},
     vol~156, pp.~1--24, 2013 (Invited paper).





 \bibitem{BY2}  J.M.\ Borwein and L.\ Yao,
``Recent progress on  Monotone Operator Theory'',
\emph{ Infinite Products of Operators and Their Applications}, Contemporary Mathematics, in press;\\
   \url{http://arxiv.org/abs/1210.3401v3}.


 \bibitem{BY3}  J.M.\ Borwein and L.\ Yao,
``Maximality of the sum of a maximally monotone linear relation and a
maximally monotone operator'', \emph{Set-Valued and Variational Analysis},
 vol.~21, pp.~603--616, 2013.


\bibitem{BY4FV}  J.M.\ Borwein and L.\ Yao,
``Sum theorems  for maximally monotone operators of type (FPV)'', \emph{Journal of the Australian Mathematical Society}, in press;\\
\texttt{http://arxiv.org/abs/1305.6691v1}.


\bibitem{BurIus}
R.S.\ Burachik and A.N.\ Iusem,
\emph{Set-Valued Mappings and Enlargements of Monotone Operators},
Springer-Verlag, 2008.


\bibitem{ButIus}
D.\ Butnariu and A.N.\ Iusem,
\emph{Totally Convex Functions for Fixed Points Computation
and Infinite Dimensional Optimization},
Kluwer Academic Publishers, 2000.




\bibitem{Fitz88}
S.\ Fitzpatrick,
``Representing monotone operators by convex
functions'', in  \emph{Workshop/Miniconference on Functional Analysis
and Optimization (Canberra 1988)}, Proceedings of the Centre for
Mathematical Analysis, Australian National University, vol.~20,
Canberra, Australia, pp.~59--65, 1988.


\bibitem{FitzPh}
S.P.\ Fitzpatrick and R.R.\ Phelps,
``Some properties of maximal monotone operators on nonreflexive Banach spaces'',
  \emph{Set-Valued  Analysis},  vol.~3,
 pp.~51--69, 1995.

\bibitem{MarSva5} M.\  Marques Alves and B.F.\ Svaiter,
``A new qualification condition for the maximality of the sum of
maximal monotone operators in general Banach spaces'', \emph{Journal
of Convex Analysis}, vol.~19, pp.~575--589, 2012.

\bibitem{Megg}
R.E.\  Megginson,
\emph{An Introduction to Banach Space Theory},
Springer-Verlag, 1998.

\bibitem{ph}
R.R.\ Phelps,
\emph{Convex Functions, Monotone Operators and
Differentiability},
2nd Edition, Springer-Verlag, 1993.


%\bibitem{ph2}
%R.R.\ Phelps,
%``Lectures on maximally monotone operators'', \emph{Extracta Mathematicae}, vol.~12, pp.~193--230, 1997;\\
%\texttt{http://arxiv.org/abs/math/9302209v1}, February 1993.


\bibitem{Rock66}
R.T.\ Rockafellar,
``Extension of Fenchel's duality theorem for
convex functions'',
\emph{Duke Mathematical Journal}, vol.~33, pp.~81--89, 1966.


\bibitem{Rock70CA}
 R.T.\ Rockafellar, \emph{Convex Analysis},
  Princeton Univ. Press, Princeton, 1970.

\bibitem{Rock69}
R.T.\ Rockafellar,
``Local boundedness of nonlinear, monotone operators",
\emph{Michigan Mathematical Journal}, vol.~16, pp.~397--407, 1969.



\bibitem{Rock70}
R.T.\ Rockafellar,
``On the maximality of sums of nonlinear monotone operators'',
\emph{Transactions of the American Mathematical Society},
vol.~149, pp.~75--88, 1970.




\bibitem{RockWets}
R.T.\ Rockafellar and R.J-B Wets,
\emph{Variational Analysis}, 3nd Printing,
Springer-Verlag, 2009.


\bibitem{Rudin}
R.\ Rudin,
\emph{Functional Analysis},
Second Edition, McGraw-Hill, 1991.



\bibitem{Si}
S.\  Simons,
\emph{Minimax and Monotonicity},
Springer-Verlag, 1998.

\bibitem{Si2}
S.\ Simons, \emph{From Hahn-Banach to Monotonicity},
% Lecture Notes in Mathematics, Vol. 1693,
Springer-Verlag, 2008.


\bibitem{SiZ}
 S.\ Simons and C.\  Z{\v{a}}linescu, ``Fenchel duality, Fitzpatrick functions
  and maximal monotonicity'',
  \emph{Journal of Nonlinear and Convex Analysis}, vol.~6, pp. 1--22, 2005.


\bibitem{VV5}
A.\ Verona and M.E.\ Verona,
``Remarks on subgradients and $\varepsilon$-subgradients'',
\emph{Set-Valued  Analysis},
vol.~1, pp.~261--272, 1993.

\bibitem{VV1}
A.\ Verona and M.E.\ Verona,
``Regular maximal monotone operators'',
\emph{Set-Valued  Analysis},
vol.~6, pp.~303--312, 1998.



\bibitem{VV2}
A.\ Verona and M.E.\ Verona,
``On the regularity of maximal monotone operators and related results'';
 \url{http://arxiv.org/abs/1212.1968v3}, December 2012.




\bibitem{Voi1}
M.D.\ Voisei,
``The sum and chain rules for maximal monotone operators",
\emph{Set-Valued and Variational Analysis}, vol.~16, pp.~461--476, 2008.




\bibitem{ZalVoi}
M.D.\ Voisei and C.\ Z\u{a}linescu,
``Maximal monotonicity criteria for the composition and the sum under weak interiority conditions'',
\emph{Mathematical Programming (Series B)},
vol.~123, pp.~265--283, 2010.


\bibitem{Yao3}
L.\ Yao,  ``The sum of a maximal monotone operator  of type (FPV) and a maximal monotone operator
with full domain is maximally monotone'', \emph{Nonlinear Analysis}, vol.~74, pp.~6144--6152, 2011.


\bibitem{Yao2}
L.\ Yao,  ``The sum of a maximally monotone
linear relation and the subdifferential of a proper lower semicontinuous
convex function is maximally monotone'',
\emph{Set-Valued and Variational Analysis}, vol.~20, pp.~155--167, 2012.

\bibitem{YaoPhD}
L.\ Yao,
\emph{On Monotone Linear Relations and the Sum Problem in Banach
Spaces}, Ph.D. thesis, University of British Columbia
(Okanagan), 2011; \url{http://hdl.handle.net/2429/39970}.



\bibitem{Zalinescu}
{C.\ Z\u{a}linescu},
\emph{Convex Analysis in General Vector Spaces}, World Scientific
Publishing, 2002.

\bibitem{Zeidler2A}
E.\ Zeidler,
\emph{Nonlinear Functional Analysis and its Applications II/A:
Linear Monotone Operators},
Springer-Verlag, 1990.

\bibitem{Zeidler2B}
E.\ Zeidler,
\emph{Nonlinear Functional Analysis and its Applications II/B:
Nonlinear Monotone Operators},
Springer-Verlag,
1990.
\end{thebibliography}
\end{document}